\documentclass[12pt]{amsart}
\usepackage{geometry}
\usepackage[utf8]{inputenc}
\usepackage{graphicx}
\usepackage{enumitem}                               
\usepackage{amssymb}
\usepackage{mathptmx}
\usepackage{amsthm,amsfonts}

\newtheorem{theorem}{Theorem}

\newtheorem{corollary}[theorem]{Corollary}

\newtheorem{theoremn}{Theorem}[section] 
\newtheorem{lemman}[theoremn]{Lemma}

\newtheorem*{theorem*}{Theorem}
\newtheorem*{lemma*}{Lemma}

\theoremstyle{definition}
\newtheorem{definition}[theoremn]{Definition}

\newtheorem{remark}[theoremn]{Remark}

\newtheoremstyle{definition*}
{\topsep}
{\topsep}
{}
{0pt}
{\bfseries}
{.}
{ }
{\thmname{#1}\thmnumber{ #2}\thmnote{ (#3)}}

\theoremstyle{definition*}
\newtheorem*{definition*}{Definition}
\newtheorem*{remark*}{Remark}
\newtheorem*{claim*}{Claim}
\newtheorem*{conclusion*}{Conclusion}

\newcommand{\Qdp}{\text{Qd}(p)}

\begin{document}
	\baselineskip 13.75pt
	
	\title[ ]{An extension of the Glauberman ZJ-Theorem }

	\author{{M.Yas\.{I}r} K{\i}zmaz }
	\address{Department of Mathematics, Bilkent University, 06800 
		Bilkent, Ankara, Turkey}
	
	\email{yasirkizmaz@bilkent.edu.tr}
			\maketitle
		\begin{abstract}
			Let $p$ be an odd prime and let $J_o(X)$, $J_r(X)$ and $J_e(X)$ denote the three different versions of Thompson subgroups for a $p$-group $X$.
			In this article, we first prove an extension of Glauberman's replacement theorem  (\cite[Theorem 4.1]{Gla}).  Secondly, we prove the following: Let $G$ be a $p$-stable group and $P\in Syl_p(G)$. Suppose that $C_G(O_{p}(G))\leq O_{p}(G)$. If $D$ is a strongly closed subgroup in $P$, then $Z(J_o(D))$, $\Omega(Z(J_r(D)))$ and $\Omega(Z(J_e(D)))$ are normal subgroups of $G$. Thirdly, we show the following: Let $G$ be a $\Qdp$-free group and $P\in Syl_p(G)$. If $D$ is a strongly closed subgroup in $P$, then the normalizers of the subgroups $Z(J_o(D))$, $\Omega(Z(J_r(D)))$ and $\Omega(Z(J_e(D)))$ control strong $G$-fusion in $P$. We also prove a similar result for a $p$-stable and $p$-constrained group. Lastly, we give a $p$-nilpotency criteria, which is an extension of Glauberman-Thompson $p$-nilpotency theorem.
		\end{abstract}
		
		\section{Introduction}
		Throughout the article, all groups considered are finite.
		Let $P$ be a $p$-group. For each abelian subgroup $A$ of $P$, let $m(A)$ be the rank of $A$, and let $d_r(P)$ be the maximum of the numbers $m(A)$. Similarly, $d_o(P)$ is defined to be the maximum of orders of abelian subgroups of $P$ and $d_e(P)$ is defined to be the maximum of orders of elementary abelian subgroups of $P$. Define $$\mathcal A_r(P)=\{A\leq P \mid A \textit{ is abelian} \textit{ and } m(A)=d_r(P) \}, $$
		$$\mathcal A_o(P)=\{A\leq P \mid A \textit{ is abelian} \textit{ and } |A|=d_o(P) \}$$
		and
		$$\mathcal A_e(P)=\{A\leq P \mid A \textit{ is elementary abelian} \textit{ and } |A|=d_e(P) \}.$$
		Now we are ready to define three different versions of Thompson subgroup:
		$J_r(P)$, $J_o(P)$ and $J_e(P)$ are subgroups of $P$ generated by all members of $\mathcal A_r(P), \mathcal A_o(P)$ and $\mathcal A_e(P)$, respectively.
		
		Thompson  proved his normal complement theorem according to $J_r(P)$ in \cite{Thmp}, which states that ``if $N_G(J_r(P))$ and $C_G(Z(P))$ are both $p$-nilpotent and $p$ is odd then $G$ is $p$-nilpotent". Later Thompson introduced ``a replacement theorem" and a subgroup similar to $J_o(P)$ in \cite{Thmp2}. Due to the compatibility of the replacement theorem with $J_o(P)$, Glauberman worked with $J_o(P)$, indeed, he extended the replacement theorem of Thompson for odd primes (see \cite[Theorem 4.1]{Gla}). We should note that Glauberman's replacement theorem is one of the important ingredients of the proof of his ZJ-theorem.\\\\
		\textbf{Theorem (Glauberman).} Let $p$ be an odd prime, $G$ be a $p$-stable group, and $P\in Syl_p(G)$. Suppose that $C_G(O_{p}(G))\leq O_{p}(G)$. Then $Z(J_o(P))$
		is a characteristic subgroup of $G$.\\\\
		There are many important consequences of his theorem. One of the striking ones is that $N_G(Z(J_o(P)))$ controls strong $G$-fusion in $P$ when $G$ does not involve a subquotient isomorphic to $Q_d(p)$ (see \cite[Theorem B]{Gla}). Another consequence of his theorem is an improvement of Thompson normal complement theorem. This result says that  if $N_G(Z(J_o(P))$ is $p$-nilpotent and $p$ is odd then $G$ is $p$-nilpotent. 
		
		There is still active research on properties of Thompson's subgroups. A current article \cite{Pt} is describing algorithms for determining $J_e(P)$ and $J_o(P)$. We also refer to \cite{Pt} and \cite{Kho} for more extensive discussions about literature and replacement theorems, which we do not state here. It deserves to be mentioned separately that Glauberman obtained remarkably more general versions of the Thompson replacement theorem  in his later works (see \cite{Gla4} and \cite{Gla3}). We should also note that even if \cite[Theorem 1]{Pt} is attributed to Thompson replacement theorem \cite{Thmp} in \cite{Pt}, it seems that the correct reference is Isaacs replacement theorem (see \cite{Isc2}).
		
		In \cite{Crv}, the ZJ-theorem is given according to $J_e(P)$ (see \cite[Theorem 1.21, Definition 1.16]{Crv}). Although it might be natural to think that Glauberman ZJ-theorem is also correct for ``$J_e(P)$ and $J_r(P)$", there is no reference verifying that. We should also mention that Isaacs proved the Thompson normal complement theorem according to $J_e(P)$ in his book (see \cite[Chapter 7]{Isc}). However, the ZJ-theorem is not contained in his book.
		
		One of the purposes of this article is to generalize Glauberman replacement theorem (see \cite[Theorem 4.1]{Gla}), which was used by Glauberman in the proof of his ZJ-theorem. We also note that our replacement theorem is an extension of Isaacs replacement theorem (see \cite{Isc2}) when we consider odd primes. The following is the first main theorem of our article:
		\begin{theorem}\label{A}
			Let $G$ be a $p$-group for an odd prime $p$ and $A\leq G$ be abelian. Suppose that $B\leq G$ is of class at most $2$ such that $B'\leq A$, $A\leq N_G(B)$ and $B\nleq N_G(A)$. Then there exists an abelian subgroup $A^*$ of $G$ such that
			\begin{enumerate}[label=(\alph*)]
				\item\textit{ $|A|=|A^*|,$}
				\item  \textit{$A\cap B <A^*\cap B,$}
				\item \textit{$A^*\leq N_G(A)\cap A^G,$}
				\item \textit{ the exponent of $A^*$ divides the exponent of $A$. Moreover, $rank(A)\leq rank(A^*)$.}
			\end{enumerate}
		\end{theorem}
		One of the main differences from \cite[Theorem 4.1]{Gla} is that we are not taking $A$ to be of maximal order. By removing the order condition, we obtain more flexibility to apply the replacement theorem. Since our replacement theorem is easily applicable to all versions of Thompson subgroups and there is a gap in the literature whether ZJ-theorem holds for other versions of Thompson subgroups, we shall prove our extensions of ZJ-theorem for all different versions of Thompson subgroups.
		\vspace{0.15cm}
		 \begin{definition*}\cite[pg 22]{Gla2}\label{def:p-stable}
			A group $G$ is called \textbf{$p$-stable} if it satisfies the following condition: Whenever $P$ is a $p$-subgroup 
			of $G$, $g\in N_G(P)$ and $[P,g,g]=1$ then the coset $gC_G(P)$ lies in $O_p(N_G(P)/C_G(P))$. 
		\end{definition*}
	\vspace{0.15cm}
		Let $K$ be a $p$-group. We write  $\Omega(K)$ to denote the subgroup $\langle \{x\in K\mid x^{p}=1 \}\rangle$ of $K$. Note that $\Qdp$ is defined to be a semidirect product of $\mathbb Z_p \times \mathbb Z_p$ with $SL(2,p)$ by the natural action of $SL(2,p)$ on $\mathbb Z_p \times \mathbb Z_p$. Here is the second main theorem of the article;
		\begin{theorem}\label{B}
			Let $p$ be an odd prime, $G$ be a $p$-stable group, and $P\in Syl_p(G)$. Suppose that $C_G(O_{p}(G))\leq O_{p}(G)$. If $D$ is a strongly closed subgroup in $P$ then $Z(J_o(D)), \ \Omega(Z(J_r(D)))$ and $\Omega(Z(J_e(D)))$ are normal subgroups of $G$.
		\end{theorem}
		We prove Theorem B  mainly by following the original proof given by Glauberman and with the help of Theorem A. When we take $D=P$, we obtain that $Z(J_o(P)),\Omega(Z(J_r(P)))$ and $\Omega(Z(J_e(P)))$ are characteristic subgroups of $G$ under the hypothesis of Theorem B. Both $Z(J_r(P))$ and $Z(J_e(P))$ need an extra operation  ``$\Omega$" and it does not seem quite possible to remove ``$\Omega$" by the method used here.
		\begin{corollary}\label{C}
Let $p$ be an odd prime, $G$ be a $p$-stable group, and $P\in Syl_p(G)$. Suppose that $C_G(O_{p}(G))\leq O_{p}(G)$ and $D$ is a strongly closed subgroup in $P$. If the exponent of $\Omega(D)$ is $p$, then $Z(J_o(\Omega(D))), \ Z(J_r(\Omega(D)))$ and $Z(J_e(\Omega(D)))$ are normal subgroups of $G$.
		\end{corollary}
	
		\begin{proof}[\textbf{Proof}]
			Suppose that the exponent of $\Omega(D)$ is $p$. Let $U\leq \Omega(D)$ and $U^g\leq P$ for some $g\in G$. Then we see that $U^g \leq D$ as $D$ is strongly closed in $P$. Since the exponent of $U$ is $p$, we get that $U^g\leq \Omega(D)$. Thus $\Omega(D)$ is strongly closed in $P$, and so $Z(J_o(\Omega(D)))\lhd G$ by Theorem B. On the other hand, $J_o(\Omega(D))=J_e(\Omega(D))=J_r(\Omega(D))$ since the exponent of $\Omega(D)$ is $p$. Then the result follows.
		\end{proof}
		
		Note that the condition on the exponent of $\Omega(D)$ is naturally satisfied if $\Omega(D)$ is a regular $p$-group and it is well known that $p$-groups of class at most $p-1$ are regular. Thus, we may apply Corollary C when $|\Omega(D)|\leq p^p$, in particular. One of the advantages of working with $\Omega(D)$ is that $J_x(\Omega(D))$ could be determined more easily  compared to $J_x(D)$ for most of the $p$-groups for $x\in \{o,r,e\}$.
	\vspace{0.1cm}
	\begin{definition*}\cite[pg 268]{Gor}
	A group $G$ is called \textbf{	$p$-constrained} if $C_G(U)\leq O_{p',p}(G)$ for a Sylow $p$-subgroup $U$ of $O_{p',p}(G)$.
\end{definition*}

\begin{theorem}\label{E}
	Let $p$ be an odd prime, $G$ be a $p$-stable group, and $P\in Syl_p(G)$. Assume that $N_G(U)$ is $p$-constrained for each nontrivial subgroup $U$ of $P$. If $D$ is a strongly closed subgroup in $P$ then the normalizers of the subgroups $Z(J_o(D))$, $\Omega(Z(J_r(D)))$ and $\Omega(Z(J_e(D)))$ control strong $G$-fusion in $P$.
\end{theorem}

\begin{remark}
	In \cite{H}, it is proven that if $G$ is $p$-stable and $p>3$ then $G$ is $p$-constrained by using classification of finite simple groups (see Proposition 2.3 in \cite{H}). Thus, the assumption ``$N_G(U)$ is $p$-constrained for each nontrivial subgroup $U$ of $P$'' is automatically satisfied when $p>3$ and $G$ is a $p$-stable group.
\end{remark}

	\begin{theorem}\label{F}
Let $p$ be an odd prime, $G$ be a \Qdp-free group, and $P\in Syl_p(G)$. If $D$ is a strongly closed subgroup in $P$ then the normalizers of the subgroups $Z(J_o(D))$, $\Omega(Z(J_r(D)))$ and $\Omega(Z(J_e(D)))$   control strong $G$-fusion in $P$.
	\end{theorem}

		\begin{remark}
			In Theorem \ref{F}, if we take $D=P$, then the proof of this special case follows by Theorem \ref{B} and \cite[Theorem 6.6]{Gla2}. However, the general case requires some extra work. Indeed, we shall prove Theorem \ref{F} by constructing an appropriate section conjugacy functor depending on $D$, and applying \cite[Theorem 6.6]{Gla2}.
		\end{remark}
		The following is an easy corollary of Theorem \ref{F}.
		\begin{corollary}
Let $p$ be an odd prime, $G$ be a $\Qdp$-free group, and $P\in Syl_p(G)$. If the exponent of $\Omega(D)$ is $p$, then the normalizers of the subgroups $Z(J_o(\Omega(D))),Z(J_r(\Omega(D)))$ and $Z(J_e(\Omega(D)))$ control strong $G$-fusion in $P$.
		\end{corollary}
		
		\begin{proof}[\textbf{Proof}]
			As in the proof of Corollary \ref{C}, we see that $\Omega(D)$ is strongly closed in $P$ since the exponent of $\Omega(D)$ is $p$. Thus, $J_o(\Omega(D))=J_r(\Omega(D))=J_e(\Omega(D))$ and the result follows by Theorem \ref{F}.
		\end{proof}

		Lastly we state an extension of Glauberman-Thompson $p$-nilpotency theorem.
		\begin{theorem}\label{H}
	Let $p$ be an odd prime, $G$ be a group and $P\in Syl_p(G)$. If $D$ is a strongly closed subgroup in $P$ then $G$ is $p$-nilpotent if one of the normalizer of subgroups $Z(J_o(D))$, $\Omega(Z(J_r(D)))$ and $\Omega(Z(J_e(D)))$ is $p$-nilpotent.
		\end{theorem}
		
		\section{The proof of theorem A}
		We first state the following lemma, which is extracted from the proof of Glauberman replacement theorem. 
		
		\begin{lemman}[Glauberman]\label{Glb}
			Let $p$ be an odd prime and $G$ be a $p$-group. Suppose that $G=BA$ where $B$ is a normal subgroup of $G$ such that $B'\leq Z(G)$ and $A$ is an abelian subgroup of $G$ such that $[B,A,A,A]=1$. Then $[b,A]$ is an abelian subgroup of $G$ for each $b\in B$.
		\end{lemman}
		
		\begin{proof}[\textbf{Proof}]
			Let $x,y\in A$. Our aim is to show that $[b,x]$ and $[b,y]$ commute. Set $u=[b,y]$. If we apply Hall-Witt identity to the triple $(b,x^{-1},u)$, we obtain that
			$$[b,x,u]^{x^{-1}}[x^{-1},u^{-1},b]^{u}[u,b^{-1},x^{-1}]^b=1.$$
			
			Note that the above commutators of weight $3$ lie in the center of $G$ since $B'\leq Z(G)$. Thus we may remove conjugations in the above equation. Moreover, $[u,b^{-1},x^{-1}]=1$ as $[u,b^{-1}]\in B'$. Thus we obtain that $[b,x,u][x^{-1},u^{-1},b]=1$, and so
			$$[b,x,u]=[x^{-1},u^{-1},b]^{-1}.$$
			
			Since $[x^{-1},u^{-1},b]=[[x^{-1},u^{-1}],b]\in Z(P)$, we see that
			$$[x^{-1},u^{-1},b]^{-1}=[[x^{-1},u^{-1}],b]^{-1}=[[x^{-1},u^{-1}]^{-1},b]=[[u^{-1},x^{-1}],b]$$ by \cite[Lemma 2.2.5(ii)]{Gor}. As a consequence, we get that $[b,x,u]=[[u^{-1},x^{-1}],b]$. By inserting $u=[b,y]$, we obtain $$[[b,x],[b,y]]=[[[b,y]^{-1},x^{-1}],b].$$
			
			Now set $\overline G=P/B'$. Then clearly $\overline B$ is abelian. It follows that $[\overline B,\overline A, \overline A]\leq Z(\overline P)$ since $[B,A,A,A]=1$ and $\overline B$ is abelian. Then we have
			$$[[b,y]^{-1},x^{-1}]\equiv [[b,y]^{-1},x]^{-1}\equiv [[b,y],x] \ mod \ B' $$ by applying \cite[Lemma 2.2.5(ii)]{Gor} to $\overline G$. Since $x$ and $y$ commute and $\overline {[b,A]}\subseteq \overline B$ is abelian, we see that $$[b,y,x]\equiv [b,x,y] \ mod \ B'$$
			by \cite[Lemma 2.2.5(i)]{Gor}.
			
			Finally we obtain $$[[b,x],[b,y]]=[[[b,y]^{-1},x^{-1}],b]= [[[b,y],x],b]=[[b,x,y],b].$$
			
			By symmetry, we also have that $[[b,y],[b,x]]=[[b,x,y],b]$. Then it follows that $[[b,y],[b,x]]=[[b,y],[b,x]]^{-1}$, and so $[[b,x],[b,y]]=1$ since $G$ is of odd order.

		\end{proof}
		\begin{lemman}\label{rank lemma}
		Let $A$ be an abelian $p$-group and $E$ be the largest elementary abelian subgroup of $A$. Then $rank(E)=rank(A)$.
	\end{lemman}
	\begin{proof}[\textbf{Proof}]
		Consider the homomorphism $\phi:A\to A$ by $\phi(a)=a^p$ for each $a\in A$. Notice that $\phi(A)=\Phi(A)$ and $E=Ker(\phi)$, and so $|A/\Phi(A)|=|E|$. Since both $E$ and $A/\Phi(A)$ are elementary abelian groups of same order, we get $rank(E)=rank(A/\Phi(A))$. On the other hand, $rank(A/\Phi(A))=rank(A)$ and the result follows.
	\end{proof}
		
		\begin{proof}[\textbf{Proof of Theorem A}]
			We proceed by induction on the order of $G$. We can certainly assume that $G=AB$. Since $A$ is not normal in $G$, there exists a maximal subgroup $M$ of $G$ such that $A\leq M$.
			
			Clearly $A$ normalizes $M\cap B$ as both $M$ and $B$ are normal in $G$. Suppose that $M\cap B$ does not normalize $A$. By induction applied to $M$, there exists a subgroup $A^*$ of $M$ such that $A^*$ satisfies the conclusion  of the theorem. Then $A^*$ also satisfies $(a),\ (c)$ and $(d)$ in $G$. Moreover, $A\cap (M\cap B)=A\cap B<  A^*\cap B$, and so $G$ also satisfies the theorem. Hence, we can assume that $M\cap B\leq N_G(A)$. Notice that $M=M\cap AB=A(M\cap B)$, and so $M=N_G(A)$.
			
			Clearly $M\cap B$ is a maximal subgroup of $B$. Then $A$ acts trivially on $B/(M\cap B)$, and so $[B,A]\leq M=N_G(A)$. Thus, we see that $[B,A,A]\leq A$ which yields $[B,A,A,A]=1.$ Moreover, we have that $B'\leq Z(G)$ since $B'\leq A$ and $B'\leq Z(B)$. It follows that $[b,A]$ is abelian for any $b\in B$ by Lemma \ref{Glb}.
			
			Let $b\in B\setminus M$. Then $A\neq A^b\lhd M$. Set $H=AA^b$ and $Z=A\cap A^b$. Then clearly $H$ is a group and $Z\leq Z(H)$. On the other hand, $H$ is of class at most $2$ since $H/Z$ is abelian. Note that the identity $(xy)^n=x^ny^n[x,y]^{\frac{n(n-1)}{2}}$ holds for all $x,y\in H$ as $H$ is of odd order. It follows that the exponent of $H$ is  the same as the exponent of $A$. 
			
			Now we shall show that $H\cap B$ is abelian. First we claim that $H\cap B=(A\cap B)[A,b]$.
			Clearly, we have $[A,b]\subseteq H\cap B$ since $H=AA^b$. It follows that $(A\cap B)[A,b]\subseteq H\cap B$ as $A\cap B \leq H\cap B$. Next we obtain the reverse inequality.
			Let $x\in H\cap B$. Then $x=ac^b$ for $ a,c\in A$ such that $ac^b\in B$.
			Since $B\lhd G$, we see that $[c,b]\in B$, and so $ac\in B$ as $ac[c,b]=ac^b\in B$.
			 It follows that $ac\in A\cap B$ and $x=ac[c,b]\in (A\cap B)[A,b]$, which proves the equality. Since $B'\leq A$, we see that $A\cap B\lhd B$. Then $A\cap B=A^b\cap B$ and hence $A\cap B=Z\cap B$. In particular, we see that $A\cap B\leq Z\leq Z(H)$. It follows that $H\cap B=(A\cap B)[A,b]$ is abelian since $[A,b]$ is an abelian subgroup of $H$ and $(A\cap B)\leq Z(H)$.
			
			Now set $A^*=(H\cap B)Z$. Note that $A^*$ is abelian as $H\cap B$ is abelian and $Z\leq Z(H)$. Now we shall show that $A^*$ is the desired subgroup.
			Clearly, the exponent of $A^*$ divides the exponent of $H$, which shows the first part of $(d)$. Note that $A<H$ and $H=H\cap AB=A(H\cap B)$, and so $H\cap B>A\cap B$. It follows that $A^*\cap B\geq H\cap B>A\cap B$, which shows $(b)$. On the other hand, $$A^*\leq H=AA^b\leq M\cap A^G=N_G(A)\cap A^G,$$ which shows $(c)$. It remains to prove $(a)$ and the second part of $(d)$.
			Since $A^*=(H\cap B)Z$, we have
			$$|A^*|=\dfrac{|H\cap B||Z|}{|Z\cap B|}=\dfrac{|H\cap B||Z|}{|A\cap B|}.$$ 
			On the other hand, $H=AA^b=A(H\cap B)$. Hence we have 
			$$\dfrac{|A A^b|}{|A^b|}=\dfrac{|A|}{|A\cap A^b|}=\dfrac{|A|}{|Z|}=\dfrac{|H\cap B|}{|A\cap B|}.$$
			
			Thus, we see that $|A|=|A^*|$ as desired.
			
			 Now let $E$ be the largest elementary abelian subgroup of $A$. We shall observe that $E$ and $A$ enjoy some similar properties. Note that $E\lhd M=N_G(A)$ since $E$ is a characteristic subgroup of $A$. Hence, $EE^b$ is a group. 
			Now set $H_1=EE^b$, $Z_1=E\cap E^b$ and $E^*=(H_1\cap B)Z_1$. First observe that $Z_1\leq Z(H_1)$, and so $H_1$ is of class at most $2$. It follows that the exponent of $E^*$ is $p$ since $H_1$ is of odd order. Thus, $E^*$ is elementary abelian as $E^*\leq A^*$ and $A^*$ is abelian. Note also that $E\cap B=E\cap (A\cap B)$, and so $E\cap B$ is characteristic in $A\cap B$. Then we see that $E\cap B\lhd B$ as $A\cap B\lhd B$. This also yields that $E\cap B=(E\cap B)^b=E^b\cap B$, and hence $E\cap B=Z_1\cap B$. Lastly, observe that $H_1=EE^b=EE^b\cap EB=E(H_1\cap B)$. Now we can show that $|E|=|E^*|$ by using the same method used for showing that  $|A|=|A^*|$. Then we see that $ rank(A)=rank(E)=rank(E^*)\leq rank(A^*)$ by Lemma \ref{rank lemma}.

		\end{proof}
		
		\section{The proof of theorem \ref{B}}

		\begin{lemman}Let $P$ be a $p$-group and $R$ be a subgroup of $P$. Then if there exists $A\in \mathcal A_{x}(P)$ such that $A\leq R$ then $J_x(R)\leq J_x(P)$ for $x\in \{ o,r,e\}$.  Moreover, $J_x(P)=J_x(R)$ if and only if $J_x(P)\subseteq R$ for $x\in \{ o,r,e\}$.
		\end{lemman}
		The above lemma is an easy observation and we shall use it without any further reference.
		
		\begin{lemman}\cite[Theorem 8.1.3]{Gor} \label{opg}
			Let $G$ be a $p$-stable group such that $C_G(O_p(G))\leq O_p(G)$. If $P\in Syl_p(G)$ and $A$ is an abelian normal subgroup of $P$ then $A\leq O_p(G)$.
		\end{lemman}
		
		\begin{proof}[\textbf{Proof}]
			Since $O_p(G)$ normalizes $A$, we see that $[O_p(G),A,A]=1$. Write $C=C_G(O_p(G))$. Then  we have $AC/C\leq O_p(G/C)$.  Note that $O_p(G/C)=O_p(G)/C $ since $C\leq O_p(G)$. It follows that $A\leq O_p(G)$.
		\end{proof}
		
		\begin{definition}\label{strongly closed set}
			Let $G$ be a group,  $P\in Syl_p(G)$ and  $D$ be a nonempty subset of $P$. We say that $D$  is a strongly closed subset in $P$ (with respect to $G$) if for all $U\subseteq D$ and $g\in G$ such that $U^g\subseteq P$, we have $U^g\subseteq D$. 
		\end{definition}
		
		\begin{lemman}\label{strogly closed}
			Let $G$ be a group and $P\in Syl_p(G)$.	Suppose that $D$ is a strongly closed subset in $P$. If $N\lhd G$ and $D\cap N$ is nonempty then $D\cap N$ is  also a strongly closed subset in $P$. Moreover, $G=N_G(D\cap N)N$.
		\end{lemman}
		
		\begin{proof}[\textbf{Proof}]
			Let $Q=P\cap N$ and write $D^*=D\cap N$. Then we see that $Q\in Syl_p(N)$. Let $U\subseteq D^*$ and $g\in G$ such that $U^g\subseteq P$. It follows that $U^g\subseteq D$ as $U\subseteq D$ and $D$ is strongly closed in $G$. Since $N\lhd G$, we see that $U^g\leq N$ which yields that $U^g\subseteq N\cap D=D^*$ which shows the first part. 
			
			We already know that $G=N_G(Q)N$ by Frattini argument. Thus, it is enough to show that $N_G(Q)\leq N_G(D^*)$. Let $x\in N_G(Q)$. Then $D^{*^x}\subseteq Q\leq P$. Since $D^*$ is strongly closed in $P$, we see that $D^*{^x}= D^*$. It follows that $x\in N_G(D^*)$, as desired.
		\end{proof}
		
		\begin{lemman}\label{crucial lemma}
			Let $P$ be a $p$-group, $p$ be odd, and  let $B,N\unlhd P$. Suppose that $B$ is of class at most $2$ and $B'\leq A$ for all $A\in \mathcal A_x(N)$. Then there exists $A\in \mathcal  A_x(N)$ such that $B$ normalizes $A$ while $x\in \{o,r,e\}.$
		\end{lemman}
		
		\begin{proof}[\textbf{Proof}]
			First suppose that $x=e$. Now choose $A\in \mathcal A_e(N)$ such that $A\cap B$ is maximum possible. If $B$ does not normalize $A$ then there exists an  abelian subgroup $A^*\leq P$ such that $|A^*|=|A|$, $A^*\leq A^P\cap N_P(A)$, $A^*\cap B>A\cap B$ and the exponent of $A^*$ divides that of $A$ by Theorem A. We first observe that $A^*$ is an elementary abelian subgroup as the exponent of $A$ is $p$. Since $A\leq N\lhd P$, we see that $A^*\leq A^P\leq N$. Hence, $A^*\in \mathcal A_e(N)$ which contradicts to the maximality of $A\cap B$. Thus $B$ normalizes $A$ as desired.
			
			Now suppose that $x=r$ and let $ A\in \mathcal A_r(N)$. Then we apply Theorem A in a similar way and find $A^*\leq N $ with $rank(A^*)\geq rank(A)$. Since the rank of $A$ is maximal possible in $N$, we see that $A^*\in \mathcal A_r(N)$. The rest of the argument follows similarly. The case $x=o$ also  follows in a similar fashion.
		\end{proof}
		
		\begin{theoremn}\label{maim thm}
			Let $p$ be an odd prime, $G$ be a $p$-stable group, and $P\in Syl_p(G)$. Let $D$ be a strongly closed subset in $P$ and $B$ be a normal $p$-subgroup of $G$. Write $K=\langle D \rangle$, $Z_o=Z(J_o(K)), \  Z_r=\Omega(Z(J_r(K))) \textit{ and  } Z_e=\Omega(Z(J_e(K)))$. If all members of $\mathcal A_x(K)$ are included in the set $D$ then $Z_x\cap B\lhd G$ while $x\in \{o,r,e\}$. 
		\end{theoremn}
		
		\begin{proof}[\textbf{Proof}]
			Write $J(X)=J_e(X)$ for any $p$-subgroup $X$ and set $Z=Z_e$. We can clearly assume that $B\neq 1$. Let $G$ be a counter example, and choose $B$ to be the smallest possible normal $p$-subgroup contradicting to the theorem. Notice that $K\unlhd P$ as $D$ is a normal subset of $P$, and so $Z\unlhd P$. In particular, $B$ normalizes $Z$.
			
			Set $B_1=(Z\cap B)^G$. Clearly $B_1\leq B$. Suppose that $B_1<B$. By our choice of $B$, we get $Z\cap B_1\lhd G$. Since $Z\cap B\leq B_1$, we have $Z\cap B\leq Z\cap B_1\leq Z\cap B$, and hence $Z\cap B= Z\cap B_1$. This contradiction shows that $B=B_1=(Z\cap B)^G$.
			
			Clearly $B'<B$, and hence $Z\cap B' \lhd G$ by our choice of $B$. Since $Z$ and $B$ normalize each other, $[Z\cap B,B]\leq Z\cap B'$. Since $B$ and $Z\cap B'$ are both normal subgroups of $G$, we obtain $[(Z\cap B)^g,B]\leq Z\cap B'$ for all $g\in G$. This yields  $[(Z\cap B)^G,B]=[B,B]=B'\leq Z\cap B'.$ In particular, we have $B'\leq Z$, and so $[ Z\cap B, B']=1$. It follows that $[B,B']=1$ as $B=(Z\cap B )^G$. As a consequence, we see that $B$ is of class at most $2$. Notice that $Z\leq A$ for all $A\in  \mathcal A_e(K)$ due to the fact that $AZ$ is  an elementary abelian subgroup of $K$. Thus we see that, in particular, $B'\leq A$ for all $A\in \mathcal A_e(K).$
			
			Let $N$ be the largest normal subgroup of $G$ that normalizes $Z\cap B$. Set $D^*=D\cap N$, which is nonempty by our hypothesis, and write $K^*=\langle D^* \rangle$. We see that $G=N_G(D^*)N$ by Lemma \ref{strogly closed}, and so $G=N_G(K^*)N$.  It follows that $G=N_G(J(K^*))N$ since $J(K^*)$ is a characteristic subgroup of $K^*$. Suppose that $J(K)\leq K^*$. Then we see that  $J(K)=J(K^*)$, and hence $Z\cap B$ is normalized by $N_G(J(K^*))$. It follows that $Z\cap B\lhd G$. Thus we may assume that $J(K)\nsubseteq K^*.$
			
			There exists $A\in \mathcal A_e(K)$ such that $B$ normalizes $A$ by Lemma \ref{crucial lemma}. Hence, $[B,A,A]=1$ since $[B,A]\leq A$. Since $G$ is $p$-stable and $B\lhd G$, we have that $AC/C\leq O_p(G/C)$ where $C=C_G(B)$. Note that $C$ normalizes $Z\cap B$, and so $C\leq N$ by the choice of $N$. It follows that $AN/N\leq O_p(G/N).$ Now we claim that $O_p(G/N)=1$. Let $L\lhd G$ such that $L/N=O_p(G/N)$. Then $L=(L\cap P)N$, and hence $L$ normalizes $Z\cap B$ as both $N$ and $L\cap P$ normalize $Z\cap B$. The maximality of $N$ forces that $N=L$, which yields that $A\leq N$. Note that   $A\subseteq D$ by hypothesis, and so $A\subseteq N\cap D=D^*\subseteq K^*$.
			
			We see that $Z\leq A\leq J(K^*)$, and so we have $J(K^*)\leq J(K)$. It follows that $Z \cap B\leq Z\leq \Omega(Z(J(K^*)))$. Set $X=\Omega(Z(J(K^*)))$. Then we see that $G=NN_G(X)$ since $G=NN_G(K^*)$ and $X$ is characteristic in $K^*$. Since $N$ normalizes $Z\cap B$, each distinct conjugate of $Z\cap B$ comes via an element of $N_G(X).$ Thus, $B=(Z\cap B)^G=(Z\cap B)^{N_G(X)}\leq X$. 
			
			Since $J(K)\nsubseteq K^*$, some members of  $\mathcal A_e(K)$  do not lie in $ K^*$. Among such members choose $ A_1 \in \mathcal A_e(K)$ such that $A_1\cap B$ is maximum possible. Note that $B$ does not normalize $A_1$, since otherwise this forces $A_1\leq K^*$ as in previous paragraphs. Then there exists $A^*\leq P$ such that $|A^*|=|A|$, $A^*\cap B>A_1\cap B$, $A^*\leq A_1^P\cap N_P(A_1)$ and the exponent of $A^*$ divides the exponent of $A_1$ by Theorem A. Since $A_1$ is elementary abelian, we see that $A^*$ is also elementary abelian. Moreover, $A^*\leq K$ as $A_1^P\leq K\lhd P$. It follows that $A^*\in \mathcal A_e(K)$, and so $A^*\leq K^*$ due to the choice of $A_1$. We see that $XA^*$ is a group and $A^*\in \mathcal A_e(K^*)$, and hence  $B\leq X\leq A^*$. It follows that $B\leq A^*\leq N_P(A_1)$, which is the final contradiction. 
			Thus, our proof is complete for $Z_e$. Almost the same proof works for $Z_r$ and $Z_o$ without any difficulty.
		\end{proof}
		When we work with $J_o(K)$, we do not  need to use $\Omega$ operation due to the fact that $Z(J_o(K))\leq A$ for all $A\in \mathcal A_o(K)$. However, this does not need to be satisfied for $Z(J_e(K))$ and $Z(J_r(K))$. In these cases, however, the rank conditions force that $\Omega(Z(J_x(K)))\leq A$ for all $ A\in \mathcal A_{x}(K)$ for $x\in  \{e,r\}$. This difference causes the use of $\Omega$ operation necessary for $Z(J_e(K))$ and $Z(J_r(K))$.

		\begin{proof}[\textbf{Proof of Theorem \ref{B}}]
			As in our hypothesis, let $G$ be a $p$-stable group that $C_G(O_p(G))\leq O_p(G)$ and $D$ be a strongly closed subgroup in $P$. Since all these subgroups $Z(J_o(D)), \ \Omega(Z(J_r(D)))$  and  $\Omega(Z(J_e(D)))$ are abelian normal subgroups of $G$, we see that they must lie in $O_p(G)$ by Lemma \ref{opg}. Note that $D$ is also a strongly closed subset in $P$ and satisfies the hypothesis of Theorem \ref{maim thm}. Then the results follow from Theorem \ref{maim thm}.
		\end{proof}

	In this section, we see another application of Theorem \ref{maim thm} by proving the following theorem, which we shall need in the next section.
		\begin{theoremn}\label{main thm2}
		Let $p$ be an odd prime, $G$ be a $p$-stable and $p$-constrained group, and $P\in Syl_p(G)$. Let $D$ be a strongly closed subset in $P$. Write $K=\langle D \rangle$, $Z_o=Z(J_o(K)), \  Z_r=\Omega(Z(J_r(K))) \textit{ and  } Z_e=\Omega(Z(J_e(K)))$. If all members of $\mathcal A_x(K)$ are included in the set $D$, then the normalizer of $Z_x(K)$ controls strong $G$-fusion in $P$ while $x\in \{o,r,e\}$.
	\end{theoremn}
We need the following lemma in the proof of Theorem \ref{main thm2}.

\begin{lemman}\cite[Lemma 7.2]{Gla}\label{p-stable}
	If $G$ is a $p$-stable group, then $G/O_{p'}(G)$ is also $p$-stable.
\end{lemman}
Since the $p$-stability definition we used here is not same with that of \cite{Gla} and \cite[Lemma 7.2]{Gla} has also extra assumption that $O_p(G)\neq 1$, it is appropriate to give a proof of this lemma here.

\begin{proof}[\textbf{Proof}]
	Write $N=O_{p'}(G)$ and $\overline G=G/N$. Let $V$ be $p$-subgroup of $\overline G$. Then there exists a $p$-subgroup $U$ of $G$ such that $\overline U=V$.

Let $\overline x \in N_{\overline G}(\overline U)$ such that $[\overline U, \overline x, \overline x ]=\overline 1$. Clearly, we can write $\overline x=\overline x_1 \overline x_2$ such that $\overline x_1$ is a $p$-element, $\overline x_2$ is a $p'$-element  and $[\overline x_1,\overline x_2]=\overline 1$ for some $x_1,x_2\in G$. 
It follows that  $[\overline U, \overline x_i, \overline x_i ]=\overline 1$ for $i=1,2$. Then we see that $\overline x_2\in C_{\overline G}(\overline U)$ by \cite[Lemma 4.29]{Isc}. Thus, it is enough to show that $\overline x_1 \in O_p(N_{\overline G}(\overline U)/C_{\overline G}(\overline U))$ to finish the proof.

Since $\overline x_1$ is a $p$-element of $\overline G$, $x_1=sn$ where $n\in N$ and $s$ is a $p$-element of $G$, which yields that $\overline x_1=\overline s$. Then we see that $[UN,s,s]\in N$ and $s\in N_G(UN)$ by the previous paragraph. Note that $U\in Syl_p(UN)$ and $| Syl_p(UN)|$ is a $p'$-number. Consider the action of $\langle s \rangle $ on $Syl_p(UN)$. Then we observe that $s$ normalizes $U^n$ for some $n\in N$. Thus, we get that $[U^n,s,s]\leq U^n\cap N=1$. Note that $\overline U=\overline {U^n}$, and so we take $U^n=U$ without loss of generality.

Let $K\leq N_G(U)$ such that $K/C_G(U)=O_p(N_G(U)/C_G(U))$. Thus we observe that $s\in K$ as $G$ is $p$-stable. Note that $N_{\overline G}(\overline U)=\overline {N_G(U)}$ and $C_{\overline G}(\overline U)=\overline {C_G(U)}$ by \cite[Lemma 7.7]{Isc}. Hence, we see that $\overline x_1=\overline s \in \overline K$ and $\overline K/\overline {C_G(U)}\leq O_p(\overline {N_G(U)}/\overline {C_G(U)})=O_p(N_{\overline G}(\overline U)/C_{\overline G}(\overline U))$, which completes the proof.
\end{proof}

	\begin{proof}[\textbf{Proof of Theorem \ref{main thm2}}]
		Write $\overline G=G/O_{p'}(G)$. Then $\overline G$ is $p$-stable by Lemma \ref{p-stable}. Since $G$ is $p$-constrained, we have $C_{\overline G}(O_{p}(\overline G))\leq O_p(\overline G)$ by \cite[Theorem 1.1(ii)]{Gor}. Note that $Z_x\leq O_p(\overline G)$ by Lemma \ref{opg} for $x\in \{o,e,r\}$. We see that $\overline G$ satisfies the hypotheses of Theorem \ref{maim thm} as $\overline P$ is isomorphic to $P$ and $\overline D$ is the desired strongly closed set in $\overline P$. It follows that $Z_x(\overline K)\lhd \overline G$ by Theorem \ref{maim thm}, and so we get $G=O_{p'}(G)N_G(Z_x(K))$ for $x\in \{o,e,r\}$. Hence, $N_G(Z_x(K))$ controls strong $G$-fusion in $P$ by \cite[Lemma 7.1]{Gla} for $x\in \{o,e,r\}$.
	\end{proof}
			
		\section{The Proofs of Theorems \ref{E}, \ref{F} and \ref{H} }
		\begin{lemman}\label{strongly closed2}
			Let $P\in Syl_p(G)$ and $D$ be a strongly closed subset in $P$. Let $H\leq G$, $N\lhd G$ and $g\in G$ such that $P^g\cap H\in Syl_p(H)$. Then
			\begin{enumerate}[label=(\alph*)]
				\item\textit{ $D^g\cap H$ is strongly closed in $P^g\cap H$ with respect to $H$ if $D^g\cap H$ is nonempty.}
				\item  \textit{$DN/N$ is strongly closed in $PN/N$ with respect to $G/N$.}
			\end{enumerate}
		\end{lemman}
		\begin{proof}[\textbf{Proof}]
			$(a)$	Let $U\subseteq D^g\cap H$ and $h\in H$ such that $U^h\subseteq P^g\cap H$. Since $U\subseteq D^g$ and $U^h\subseteq P^g$, we see that $U^h\subseteq D^g$ as $D^g$ is strongly closed in $P^g$ with respect to $G$. Thus, $U^h\subseteq D^g\cap H$ as $U^h\subseteq H$.
			
			$(b)$ Let $U/N\subseteq DN/N$ and suppose that $(U/N)^y\subseteq PN/N$ for some $y\in G$. By an easy argument, we can find $V\subseteq D$ such that $U/N=VN/N$.
			
			Then we see that $VN\subseteq DN$ and $(VN)^y=V^yN\subseteq PN$. We need to show that $V^yN\subseteq DN$. Notice that $\langle V^y \rangle=\langle V \rangle^y $ is a $p$-subgroup of $PN$. Since $P\in Syl_p(PN)$, there exists $x\in PN$ such that $V^y\subseteq P^x$. Since $D^x$ is strongly closed in $P^x$ and $V^x\subseteq D^x$, we see that $V^y\subseteq D^x$. Thus, $V^yN\subseteq D^xN$. Write $x=mn$ for $m\in P$ and $n\in N$. Note that $D^x=D^{mn}=D^n$ as $D$ is a normal set in $P$. It follows that $D^xN=D^nN=DN$. Consequently, $V^yN\subseteq DN$ as desired.
		\end{proof}
		
		Let $\mathcal L_p(G)$ be the set of all $p$-subgroups of $G$. A map $W:\mathcal L_p(G)\to \mathcal L_p(G)$ is called \textbf{a conjugacy functor} if the followings hold for each $U\in \mathcal L_p(G)$:
		\begin{enumerate}
			\item[(i)]\textit{ $W(U)\leq U$},
			\item[(ii)]  \textit{$W(U)\neq 1 $ unless $U=1$, and}
			\item[(iii)] \textit{$W(U)^g=W(U^g)$ for all $g\in G$.}
		\end{enumerate}
		
		\textbf{A section of $G$} is a quotient group $H/K$ where $K\unlhd H\leq G$. Let $\mathcal L_p^*(G)$ be the set of all sections of $G$ that are $p$-groups. A map $W:\mathcal L_p^*(G)\to \mathcal L_p^*(G)$ is called \textbf{a section conjugacy functor} if the followings hold for each $H/K \in \mathcal L_p^*(G)$:
		\begin{enumerate}
			\item[(i)]\textit{ $W(H/K)\leq H/K$},
			\item[(ii)]  \textit{$W(H/K)\neq 1 $ unless $H/K=1$, and}
			\item[(iii)] \textit{$W(H/K)^g=W(H^g/K^g)$ for all $g\in G$.}
			\item[(iv)] Suppose that $N\lhd H$, $N\leq K$ and $K/N$ is a $p'$-group. Let $P/N$ be a Sylow $p$-subgroup of $H/N$ and set $W(P/N)=L/N$. Then $W(H/K)=LK/K$.
		\end{enumerate}
		
		For more information about section conjugacy functors and their properties, we refer to \cite{Gla2}. Note that a sufficient condition for $(iii)$ and $(iv)$ is the following: whenever $Q,R \in \mathcal L_p^*(G)$ and $\phi:Q\to R$ is an isomorphism, $\phi (W(Q))=W(R).$ Thus, the operations like $ZJ_x,\Omega ZJ_x \textit{ and } J_x$ are section conjugacy functors for $x \in \{o,r,e\}$.

		\begin{lemman}\label{conjugacy functor}
			Let $P\in Syl_p(G)$ and $D$ be a strongly closed subset in $P$. Let $W:\mathcal{L}_p(G)\to \mathcal{L}_p(G)$ be a conjugacy functor. For each $p$-subgroup $U$ of $P$ define 
			
			$$W_D(U)=\begin{cases}
			W(\langle U\cap D \rangle) &  \ if \langle U\cap D \rangle\neq 1 \\ 
			W(U) & \ if \ \langle U\cap D \rangle=1 \\ \end{cases}$$
			and for all $V\in \mathcal{L}_p(G)$ and $x\in G$ such that $V^x\leq P$ define $W_D(V)=(W_D(V^x))^{x^{-1}}.$
			Then the map $W_D:\mathcal{L}_p(G)\to \mathcal{L}_p(G)$ is a conjugacy functor. Moreover for each $y\in G$, $W_D=W_{D^y}$.
		\end{lemman}
		\begin{proof}[\textbf{Proof}]
			Since $W$ is a conjugacy functor,	it is easy to see that $W_D(U)\leq U$  and $W_D(U)\neq 1$ unless $U=1$ for each $U\in \mathcal{L}_p(G)$ by our settings.
			
			Now we need to show that $W_D(U)^g=W_D(U^g)$ for all $g\in G$ and $U\in \mathcal{L}_p(G)$, and indeed $W_D$ is well defined. First suppose that $U,U^g\leq P$ for some $g\in G$. We first show that $W_D(U)^g=W_D(U^g)$for this special case.  Note that 
			$(U\cap D)^g\subseteq U^g\leq P$, and so $(U\cap D)^g\subseteq U^g\cap D$ as $D$ is strongly closed in $P$. On the other hand, $(U^g\cap D)^{g^{-1}}\subseteq U \leq P$, and so $(U^g\cap D)^{g^{-1}}\subseteq U\cap D$ as $D$ is strongly closed in $P$. By showing the  reverse inequality, we obtain that $(U\cap D)^g= U^g\cap D.$ Now if $\langle U\cap D \rangle=1$ then $\langle U^g\cap D \rangle=1$ and  $W_D(U)^g=W(U)^g=W(U^g)=W_D(U^g)$. The second equality holds as $W$ is a conjugacy functor. On the other hand, we get $W_D(U)^g=W(\langle U\cap D \rangle)^g=W(\langle U\cap D \rangle^g)=W(\langle U^g\cap D \rangle )=W_D(U^g)$ when $\langle U\cap D \rangle \neq 1.$ 
			
			Now let $V\in \mathcal{L}_p(G)$ and $x,y\in G$ such that $V^x,V^y\leq P$. Then by setting $U=V^x$ and $g=x^{-1}y$, we have $U^g=V^y$ and $W_D(U)^g=W_D(U^g)$ by the previous paragraph. It follows that $W_D(V^y)=W_D(V^x)^{x^{-1}y}$. Then $W_D(V^y)^{y^{-1}}=W_D(V^x)^{x^{-1}}$, and so $W_D$ is well defined. Now let $z\in G$. Then $W_D(V^z)=W_D(V^x)^{x^{-1}z}=(W_D(V^x)^{x^{-1}})^z=W_D(V)^z$, which completes the proof of first part.
			
			 Lastly, since $D^y$ is strongly closed in $P^y$, $W_{D^y}$ is a conjugacy functor for $y\in G$ by the first part. It is routine to check that they are indeed the same function. 
		\end{proof}
	\begin{remark}\label{emtpty remark}
		Although a strongly closed set is nonempty according to Definition \ref{strongly closed set}, if we take $D=\emptyset$ in the previous lemma, we get $W_{\emptyset}(U)=W(U)$. Thus, we set $W_{\emptyset}=W$ for any conjugacy functor $W$.
	\end{remark}

		\begin{lemman}\label{section conjugacy functor}
			Let $P\in Syl_p(G)$ and $D$ be a strongly closed subset in $P$. Let $K\unlhd H\leq G$, $N\lhd G$ and $g\in G$ such that $P^g\cap H\in Syl_p(H)$.  Let $W:\mathcal{L}^*_p(G)\to \mathcal{L}^*_p(G)$ be a section conjugacy functor. Then the followings hold:	\begin{enumerate}[label=(\alph*)]
				\item\textit{ $W_{D^g\cap H}:\mathcal{L}_p(H)\to \mathcal{L}_p(H)$ is a conjugacy functor.}
				\item \textit{$W_{DN/N}:\mathcal{L}_p(G/N)\to \mathcal{L}_p(G/N)$ is a conjugacy functor.}
				\item \textit{$W_{(D^g\cap H)K/K}:\mathcal{L}_p(H/K)\to \mathcal{L}_p(H/K)$ is a conjugacy functor.} 
			\end{enumerate}
		\end{lemman}
		\begin{proof}[\textbf{Proof}]
			$(a)$ By taking the restrictions of $W$ to the section $H/1$, we obtain a  conjugacy functor $W:\mathcal{L}_p(H)\to \mathcal{L}_p(H)$. By Lemma \ref{strongly closed2} (a), $ D^g\cap H$ is  strongly closed in $H\cap P^g$ with respect to $H$ if $D^g\cap H$ is nonempty. Then the result follows from Lemma \ref{conjugacy functor} and Remark \ref{emtpty remark}. Similarly, $(b)$ follows by Lemma \ref{strongly closed2} (b) and Lemma \ref{conjugacy functor}. Part $(c)$ also follows in a similar fashion.
		\end{proof}
	
	\begin{remark}\label{resrection}
		 It should be noted that we only need $W$ be to be a conjugacy functor to establish  Lemma \ref{section conjugacy functor} (a).
	Now	assume the hypotheses and notation of Lemma \ref{section conjugacy functor}. Let $U\in \mathcal L_p(H)$. Then it is easy to see that $W_{D^g}(U)=W_{D^g\cap H}(U)$ by their definitions, and so $W_D(U)=W_{D^g\cap H}(U)$ by Lemma \ref{conjugacy functor}. Thus, the map $W_{D^g\cap H}$ is equal to the restriction of $W_D$ to $\mathcal L_p(H)$.
	\end{remark}
		
		\begin{lemman}\label{final}
			Assume the hypothesis and notation of Lemma \ref{section conjugacy functor}. We define $W^*_D:\mathcal{L}^*_p(G)\to \mathcal{L}^*_p(G)$ by setting $W^*_D(H/K)=W_{(D^g\cap H)K/K}(H/K)$ for each $H/K\in \mathcal{L}^*_p(G)$. Then
			$$W^*_D(H/K)=\begin{cases}
			
			W(\langle D^g\cap H \rangle K/ K) & \ if \ D^g\cap H\nsubseteq K. \\
			W(H/K) &  \ if \ D^g\cap H\subseteq K. \\

		\end{cases} $$
		Moreover, $W^*_D$ is a section conjugacy functor.
		\end{lemman}
		
		\begin{proof}[\textbf{Proof}]
			Firs suppose that $D^g\cap H\subseteq K$. Then $H/K\cap (D^g\cap H)K/K=K/K$, and so $W_{(D^g\cap H)K/K}(H/K)=W(H/K)$. If $D^g\cap H\nsubseteq K$ then $H/K\cap (D^g\cap H)K/K\neq K/K$, and so $W_{(D^g\cap H)K/K}(H/K)=W(\langle D^g\cap H \rangle K/ K)$ by its definition, which shows the first part.
			
			Note that $W^*_D(H/K)\leq H/K$ and $W^*_D(H/K)\neq 1$ unless $H/K=1$ by Lemma \ref{section conjugacy functor}(c). Now, we need to show that $(iii)$ and $(iv)$ in the definition of a section conjugacy functor hold.
			
			Pick $x\in G$. Since $(D^g\cap H)K/K$ is a  strongly subset in $(P^g\cap H)K/K$,  $(D^g\cap H)^xK^x/K^x$ is a strongly closed subset in $(P^g\cap H)^xK^x/K^x$. Moreover, $D^g\cap H\subseteq K$ if and only if $D^{gx}\cap H^x\subseteq K^x$. Thus, if $W^*_D(H/K)=W(H/K)$, then $W^*_D(H^x/K^x)=W(H^x/K^x)$. It follows that $$W^*_D(H^x/K^x)=W(H^x/K^x)=W(H/K)^x=W^*_D(H/K)^x.$$
			The second equality holds as $W$ is a section conjugacy functor.
			
			Now if $W^*_D(H/K)=W(\langle D^g\cap H \rangle K/ K)$ then $$W^*_D(H^x/K^x))=W(\langle D^{gx}\cap H^x \rangle K^x/ K^x)=W((\langle D^g\cap H \rangle K/ K)^x)=W^*_D(H/K)^x.$$
			The last equality holds as $W$ is a section conjugacy functor. Thus we see that $(iii)$ is satisfied.
			
			Now let $N\lhd H$ such that $N\leq K$ and $K/N$ is a $p'$-group. Let $X/N$ be a Sylow $p$-subgroup of $H/N$. We need to show that if $W^*_D(X/N)=L/N$ then $W^*_D(H/K)=LK/K$. Now pick $h\in H$ such that $(X/N)^h\supseteq (D^g\cap H)N/N$. By part $(iii)$, we have $W^*_D(X/N)^h=L^h/N^h=L^h/N$. If we could show that $W^*_D(H/K)=L^hK/K$, we can conclude that $$W^*_D(H/K)=W^*_D(H/K)^{h^{-1}}=(L^hK/K)^{h^{-1}}=LK/L$$ by part $(iii)$. Thus, we see that it is enough to show the claim for $(X/N)^h$, and so we may simply assume that $(D^g\cap H)N/N \subseteq X/N$.
			
			Clearly $\langle D^g\cap H\rangle$ is a $p$-group. Since $K/N$ is a $p'$-group, we see that $D^g\cap H\subseteq K$ if and only if $D^g\cap H\subseteq N$. Thus, if $W^*_D(H/K)=W(H/K)$ then $W^*_D(X/N)=W(X/N)$. It follows that $W^*_D(H/K)=LK/K$ as $W$ is a section conjugacy functor. 
			
			Assume that $D^g\cap H\nsubseteq K$. Then $W^*_D(H/K)=W(\langle D^g\cap H \rangle K/ K)$ and $W^*_D(X/N)=W(\langle D^g\cap H \rangle N/ N)=L/N$. Now write $H^*=\langle D^g\cap H\rangle K$ and $P^*=\langle D^g\cap H\rangle N$. Observe that $P^*/N\in Syl_p(H^*/N)$ and recall  $K/N$ is a $p'$-group. Since $W$ is a section conjugacy functor and  $W(P^*/N)=L/N$, we get $W(H^*/K)=LK/K$. Then the result follows.
		\end{proof}

	\begin{proof}[\textbf{Proof of Theorem \ref{E}}]
	Let $p$ be an odd prime, $G$ be a $p$-stable group and $P\in Syl_p(G)$. Suppose that $D$ is a strongly closed subgroup in $P$. Let $H$ be a $p$-constrained subgroup of $G$ and $g\in G$ such that $P^g\cap H\in Syl_p(H)$. Since each $p$-subgroup of $H$ is also a $p$-subgroup of $G$, we see that $H$ is also a $p$-stable group. 
	
	Now let $W\in \{ZJ_o, \Omega ZJ_e, \Omega ZJ_r \}$. It follows that $W_{D^g\cap H}$ is a conjugacy functor by Lemma \ref{section conjugacy functor}(a). Note that $W_{D^g\cap H}(P^g\cap H)\in \{W(D^g\cap H), W(P^g\cap H)\}$, and so $N_H(W_{D^g\cap H}(P^g\cap H))$ controls strong $H$-fusion in $P^g\cap H$ by Theorem \ref{main thm2} in both cases. Note also that $W_{D^g\cap H}(P^g\cap H)=W_D(P^g\cap H)$ by Remark \ref{resrection}.
	
	Now assume that $N_G(U)$ is $p$-constrained for each nontrivial subgroup $U$ of $P$. Fix $U\leq P$ and let $S\in Syl_p(N_G(U))$. Then by the arguments in the first paragraph, we see that the normalizer of $W_D(S)$ in $N_G(U)$ controls strong $N_G(U)$-fusion in $S$, and so we obtain that $N_G(W_D(P))$ control strong $G$-fusion in $P$ by \cite[Theorem 5.5(i)]{Gla2}. It follows that the normalizers the of the subgroups  $Z(J_o(D))$, $\Omega(Z(J_r(D)))$ and $\Omega(Z(J_e(D)))$ control strong $G$-fusion in $P$. 
	
\end{proof}
		
		\begin{lemman}\label{ff}
			Let $p$ be an odd prime, $G$ be a group, and $P\in Syl_p(G)$. Suppose that $D$ is a  strongly closed subgroup in $P$. Let $G^*$ be a section of $G$ such that $G^*$ is $p$-stable and  $C_{G^*}(O_p(G^*))\leq O_p(G^*)$. If $S\in Syl_p(G^*)$, then $W^*_D(S)\lhd G^*$ for each $W\in \{ZJ_o, \Omega ZJ_e, \Omega ZJ_r \}$. 
		\end{lemman}
		\begin{proof}[\textbf{Proof}]
			Note that $D$ is also a strongly closed set in $P$. We assume the notation of Lemma \ref{final}. Let $W\in \{ZJ_o, \Omega ZJ_e, \Omega ZJ_r \}$. Then clearly $W$ is a section conjugacy functor. It follows that $W^*_D:\mathcal L^*_p(G)\to \mathcal L^*_p(G)$ is a section conjugacy functor by Lemma \ref{final}. Let $G^*=X/K$ be a section of $G$ such that $$C_{G^*}(O_p(G^*))\leq O_p(G^*).$$
			
			Let $H/K\in Syl_p(G^*)$. Then we see that $W^*_D(H/K)=W(H/K)$ if $D^g\cap H\subseteq K$. In this case, $W(H/K)=Z(J_o(H/K)), \ \Omega (Z(J_e(H/K)))$ or $\Omega (Z(J_r(H/K))) $  which are normal subgroups of $G^*$ by Theorem \ref{B}. If $D^g\cap H\nsubseteq K$ then $(D^g\cap H)K/K$ is a strongly closed subgroup in $H/K$ with respect to $G^*$. Write $D^*=( D^g\cap H) K/K$ , then $$W^*_D(H/K)=W(D^*)=Z(J_o(D^*)), \ \Omega (Z(J_e(D^*))),  \ or \ \Omega (Z(J_r(D^*))) $$ which are normal subgroups of $G^*$ by Theorem \ref{B}. Thus we see that $W_D^*(H/K)\unlhd G^*$ for all cases.
		\end{proof}
		
		Now we are ready to prove Theorems \ref{F} and \ref{H}.
		\begin{proof}[\textbf{Proof of Theorem \ref{F}} ]
			Let $p$ be an odd prime, $G$ be a $\Qdp$-free group, and $P\in Syl_p(G)$ as in our hypothesis.
			 Since $G$ does not involve a section isomorphic to  $\Qdp$, every section of $G$ is $p$-stable by \cite[Proposition 14.7]{Gla2}. Now  let $W\in \{ZJ_o, \Omega ZJ_e, \Omega ZJ_r \}$.  Then we have that $W^*_D:\mathcal L^*_p(G)\to \mathcal L^*_p(G)$ is a section conjugacy functor by Lemma \ref{final}. Let $G^*$ be a section of $G$ such that $C_{G^*}(O_p(G^*))\leq O_p(G^*)$ and let $S\in Syl_p(G^*)$. Then we see that $W^*_D(S)\lhd G^*$ by Lemma \ref{ff}.
			 It follows that $N_G(W^*_D(P))$ controls strong $G$-fusion in $P$ by \cite[Theorem 6.6]{Gla2}. We see that $W^*_D(P)=Z(J_o(D)), \  \Omega (Z(J_e(D))),  \ or \ \Omega (Z(J_r(D)))$ according to choice of $W$, which completes the proof.
		\end{proof}
		
		\begin{proof}[\textbf{Proof of Theorem \ref{H}}]
			 Let $W\in \{ZJ_o, \Omega ZJ_e, \Omega ZJ_r \}$. Then  $W^*_D:\mathcal L^*_p(G)\to \mathcal L^*_p(G)$ is a section conjugacy functor by Lemma \ref{final}. Let $G^*$ be a section of $G$ such that $C_{G^*}(O_p(G^*))\leq O_p(G^*)$ and $G^*/(O_p(G^*)$ is $p$-nilpotent. Suppose also that $S^*\in Syl_p(G^*) $ is a maximal subgroup of $G^*$. Let $H$ be the normal Hall $p'$-subgroup of $G^*/O_p(G^*)$. Write $S=S^*/O_p(G^*)$. Then $S$ is also maximal in $G^*/(O_p(G^*)$ and $S$ acts on $H$ via coprime automorphisms. If $1<U\leq H$ is $S$-invariant then $SU=G^*/(O_p(G^*)$ by the maximality of  $S$.   Since $SH=G^*/(O_p(G^*)$ and $S\cap H=1$, we see that $U=H$. Thus, there is no proper nontrivial $S$-invariant subgroup of $H$. On the other hand, we may choose an $S$-invariant Sylow subgroup of $H$ by \cite[Theorem 3.23(a)]{Isc}. This forces $H$ to be a $q$-group for some prime $q$, and so $H'<H$. It follows that $H$ is abelian due to the fact that $H'$ is $S$-invariant.
			
			Let $H^*$ be a Hall $p'$-subgroup of $G^*$. Then we see that $H^*O_p(G^*)/(O_p(G^*)\cong H^*$. Thus, we observe that Hall $p'$-subgroups of $G^*$ are also abelian. Since $p$ is odd, we see that a Sylow $2$-subgroup of $G^*$ is abelian. This yields that $G^*$ does not involve a section isomorphic to $SL(2,p)$, and so every section of $G^*$ is $p$-stable  by \cite[Proposition 14.7]{Gla2}. Then we obtain that $W^*_D(S^*)\lhd G^*$ by Lemma \ref{ff}.
			It follows that $G$ is $p$-nilpotent by \cite[Theorem 8.7]{Gla2}.
		\end{proof}

	\section*{Acknowledgements}
	I would like to thank Prof. George Glauberman for encouraging me to study on this topic.

\end{document}